\documentclass[11pt]{article}
\usepackage{fullpage,amsthm,amssymb,tikz,amsmath,verbatim,mathtools}
\usepackage[font=footnotesize]{caption} 
\usepackage{moresize}
\def\bmm{\boldsymbol}
\def\mad{\textrm{mad}}

\def\C{\mathcal{C}}

\newtheorem{prop}{Proposition}
\newtheorem{thm}{Theorem}
\newtheorem*{thmA}{Theorem A}
\newtheorem*{thmHS}{Hajnal--Szemer\'{e}di Theorem}

\newtheorem*{mad-lem}{Mad Lemma}

\newtheorem{conj}{Conjecture}
\newtheorem{ques}{Question}
\theoremstyle{definition}
\newtheorem{defn}{Definition}
\newtheorem{clm}{Claim}[section]
\newtheorem{example}{Example}
\newtheorem*{remark}{Remark}

\usepackage[outline]{contour}%

\contourlength{0.25pt}
\contournumber{20}%

\newcommand\exampleEnd{\hfill\hbox{$\diamondsuit$}}
\newcommand\bolder[1]{\contour{black}{#1}}

\newenvironment{clmproof}[1]{\par\noindent\underline{Proof.}\space#1}{\leavevmode\unskip\penalty9999\hbox{}\nobreak\hfill\quad\hbox{$\diamondsuit$}\smallskip}

\renewcommand{\ge}{\geqslant}

\renewcommand{\le}{\leqslant}

\def\erdos{Erd\H{o}s}
\def\nesetril{Ne\v{s}et\v{r}il}

\def\tok{\textrm{tokens}}
\def\prim{\textrm{primary}}
\def\Big{\textsc{Big}}
\def\Basic{\textsc{Basic}}
\def\NonBasic{\textsc{NonBasic}}

\RequirePackage{marginnote,hyperref}
\newcommand{\aside}[1]{\marginnote{\scriptsize{#1}}[0cm]}
\newcommand{\aaside}[2]{\marginnote{\scriptsize{#1}}[#2]}
\newcommand\Emph[1]{\emph{#1}\aside{#1}}
\newcommand\EmphE[2]{\emph{#1}\aaside{#1}{#2}}
\author{
Daniel W. Cranston\thanks{%
Department of Computer Science, Virginia Commonwealth
University, Richmond, VA, USA;
\texttt{dcranston@vcu.edu}
}
\and
Gexin Yu\thanks{%
Department of Mathematics, William \& Mary, Williamsburg, VA,
USA; \texttt{gyu@wm.edu}
}
}
\begin{document}
\title{Cliques in Squares of Graphs with\\ Maximum Average Degree less than 4}
\maketitle
\abstract{
Hocquard, Kim, and Pierron constructed, for every even integer $D\ge 2$, a
2-degenerate graph $G_D$ with maximum degree $D$ such that
$\omega(G_D^2)=\frac52D$.  
We prove
for (a) all 2-degenerate graphs $G$ and (b) all graphs $G$ with $\mad(G)<4$,
upper bounds on the clique number $\omega(G^2)$ of $G^2$ that
match the lower bound given by this construction, up to small additive
constants.  We show that if $G$ is 2-degenerate
with maximum degree $D$, then $\omega(G^2)\le \frac52D+72$ (with
$\omega(G^2)\le \frac52D+60$ when $D$ is sufficiently large).  And if $G$ has
$\mad(G)<4$ and maximum degree $D$, then $\omega(G^2)\le \frac52D+532$.
Thus, the construction of Hocquard et al.~is essentially best possible.
Our proofs introduce a ``token passing'' technique to derive crucial information 
about non-adjacencies in $G$ of vertices that are adjacent in $G^2$.  
This is a powerful technique for working with such graphs that has not 
previously appeared in the literature.
}

\section{Introduction}
\label{intro-sec}

The \Emph{square, $G^2$}, of a graph $G$ is formed from $G$ by adding an edge between
each pair of vertices at distance 2 in $G$.  A graph $G$ is \emph{$k$-degenerate} if
each subgraph of $G$ has a vertex of degree at most $k$.  (Equivalently, we require that 
$G$ has a vertex order $\sigma$ such that each vertex has at
most $k$ neighbors later in $\sigma$.)
And the \emph{maximum
average degree} of $G$, denoted \Emph{$\mad(G)$} is defined as
$\mad(G):=\max_{H\subseteq G}2|E(H)|/|V(H)|$.  In this paper we study the
maximum size of a clique in $G^2$, denoted $\omega(G^2)$, over all
2-degenerate graphs $G$ and, more generally, over all graphs $G$ with
$\mad(G)<4$.  Hocquard, Kim, and Pierron~\cite{HKP} constructed 2-degenerate
graphs $G_D$ with maximum degree $D$ such that $\omega(G_D^2)=5D/2$, for every
positive even integer $D$.  This construction gives a lower bound on the
maximums that we seek.  Our two main results prove that this lower bound is
sharp, up to small additive constants.

\begin{thm}
\label{main1}
Fix a positive integer $D$.  If a graph $G$ is 2-degenerate with $\Delta(G)\le D$,
then $\omega(G^2)\le \frac52D+72$.  Furthermore, if $D\ge 1729$,
then $\omega(G^2)\le \frac52D+60$.\footnote{The remark just before Section~\ref{nice-sec} 
details subsequent improvement of this result to a version that is sharp~\cite{KL}.}
\end{thm}

\begin{thm}
\label{main2}
Fix a positive integer $D$.  If a graph $G$ has $\mad(G)<4$ and $\Delta(G)\le D$,
then $\omega(G^2)\le \frac52D+532$.
\end{thm}

Over the past 40 years tremendous effort has gone into proving upper
bounds on the chromatic number of squares of graphs; for a recent survey,
see~\cite{squares-survey}.  If a graph $G$ has maximum
degree $\Delta$, then trivially $\chi(G^2)\ge \Delta+1$ since, for every vertex
$v$, the closed neighborhood $N[v]$ is a clique in $G^2$.  
By greedily coloring vertices in the reverse of order $\sigma$ above, 
we get $\chi(G)\le k+1$.  It is easy  to check that $G^2$ is
$\Delta^2$-degenerate; thus $\chi(G^2)\le \Delta^2+1$.  This bound holds with
equality for the 5-cycle and the Petersen graph (but only for 1 or 2 other
connected graphs).

In general, the upper bound $\chi(G^2)\le \Delta^2$ cannot be improved much, as
witnessed by incidence graphs of finite projective planes.  
To prove better upper bounds on $\chi(G^2)$, researchers have focused on more
structured classes of graphs.  For brevity, we mention only two types of these.

{\erdos} and {\nesetril} introduced the notion of strong edge-coloring, which is
equivalent to coloring the square of the line graph $L(G)$.  They
conjectured that if $G$ has $\Delta\le D$, then its strong edge-chromatic
number, denoted $\chi'_s(G)$, is at most $\frac54D^2$; when $D$ is even this is
achieved by a graph $G$ with $|E(G)|=\frac54D^2$.  
So they conjectured that the worst case is when $L(G)^2$ is complete.
The current strongest bound~\cite{HdjdVK} is $\chi'_s(G)\le 1.772D^2$.
This bound is far from the conjecture, which has motivated interest in
bounding the clique number of the square of the line graph $\omega(L(G)^2)$.  
Initial
efforts in this vein~\cite{ChungGTT90} showed that if $L(G)^2$ is a complete graph,
then its order is at most $\frac54D^2$.
Later, \'{S}leszy\'{n}ska-Nowak~\cite{SN} showed that 
$\omega(L(G)^2)\le \frac32D^2$. More recently,
Faron and Postle~\cite{FP} strengthened this bound to $\omega(L(G)^2)\le \frac43D^2$.

The other class of graphs that we discuss in detail is 
planar
graphs.  In 1977, Wegner conjectured that if $G$ is planar with $\Delta\ge 8$,
then $\chi(G^2)\le \lfloor\frac32\Delta\rfloor+1$.  (His conjecture is a
chief reason for much of the interest in bounding $\chi(G^2)$.) To see that this
conjecture is sharp,
consider a so-called ``fat triangle'' formed from $K_3$ by replacing each edge
with $s$ parallel edges.  Subdivide once all edges of the fat triangle except
for two non-parallel edges.  The resulting graph $G$ has $\Delta=2s$ and $G^2$
is a complete graph of order $3s+1$.  (For smaller $\Delta$, Wegner conjectured
weaker upper bounds on $\chi(G^2)$.  In each case, the conjectured upper bound
is sharp, as witnessed by a graph  whose square is complete.)
Wegner's Conjecture has been proved asymptotically~\cite{HavetHMR}: $\chi(G^2)\le
\frac32\Delta(1+o(1))$, where the $o(1)$ is as $\Delta\to \infty$.  
But for exact bounds, the problem remains wide
open except for the case $\Delta=3$ (where 7 colors are both sufficient and
often necessary), which was confirmed by two
groups~\cite{HJT,thomassen-wegner3}.  This state of affairs motivates interest
in bounds on $\omega(G^2)$.  Amini et al.~\cite{AEvdH} proved that $\omega(G^2)\le
\frac32\Delta+O(1)$ for graphs embedded in each fixed surface.  For the plane,
they proved $\omega(G^2)\le \frac32\Delta+76$ when $\Delta\ge
11616$.
Finally, the first author proved the best possible result $\omega(G^2)\le\frac32\Delta+1$, 
when $\Delta \ge 36$.%
\footnote{The paper~\cite{AEvdH} attributes to Cohen and van den Heuvel the bound
$\omega(G^2)\le \lfloor \frac32\Delta\rfloor+1$ for all planar graphs with
$\Delta\ge 41$.  However, we contacted Cohen and 
van den Heuvel, who informed us that they have neither written a proof of this statement, nor have plans to do
so.  Thus, the problem remained open until~\cite{wegner-clique}.} 

Now we reach the specific paper that inspired our present work.  
%
Hocquard, Kim, and Pierron~\cite{HKP} studied $\chi(G^2)$ when $\mad(G)<4$.
They showed that $G^2$ is $3\Delta$-degenerate; thus,
$\chi(G^2)\le 3\Delta+1$.  They also constructed such a graph $G$ with
$\omega(G^2)=\frac52\Delta$, whenever $\Delta$ is even; see 
Example~\ref{example1}.  As a possible
strengthening, they asked the two questions below.

\begin{enumerate}
\item[(i)] Is there $D_0$ such that every 2-degenerate graph $G$ with
$\Delta(G)\ge D_0$ has $\chi(G^2)\le \frac52\Delta(G)$?
\item[(ii)] Is there $D_0$ such that every $G$ with $\Delta(G)\ge D_0$
and $\mad(G)<4$ has $\chi(G^2)\le \frac52\Delta(G)$?
\end{enumerate}

A natural approach to these questions is to try to strengthen the bound on the
degeneracy of $G^2$ whenever $G$ is 2-degenerate.  However, this is not
really possible.  In Example~\ref{example2} below, for each integer $D\ge 2$, we construct 
a family of graphs $G$ such that $G$ is 2-degenerate, $\Delta(G)=D$, and $G^2$
is \emph{not} $(3D-5)$-degenerate.  

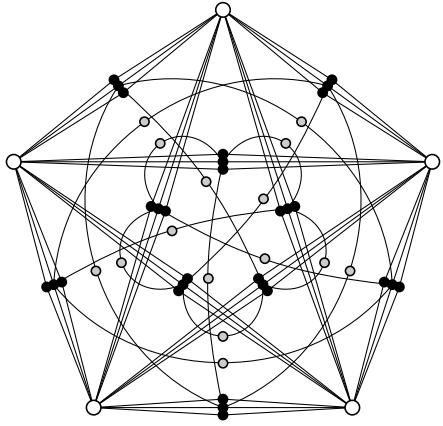
\begin{figure}[!ht]
\centering
\begin{tikzpicture}[scale=.65]
\tikzstyle{uStyle}=[shape = circle, minimum size = 5.5pt, inner sep = 0pt,
outer sep = 0pt, draw, fill=white, semithick]
\tikzstyle{lStyle}=[shape = circle, minimum size = 4.5pt, inner sep = 0pt,
outer sep = 0pt, draw, fill=none, draw=none]
\tikzstyle{usStyle}=[shape = circle, minimum size = 3.5pt, inner sep = 0pt,
outer sep = 0pt, draw, fill=black, semithick]
\tikzstyle{usGStyle}=[shape = circle, minimum size = 3.5pt, inner sep = 0pt,
outer sep = 0pt, draw, fill=gray!40!white, semithick]
\tikzset{every node/.style=uStyle}
\def\rad{4.5cm}

\foreach \i in {1,2,3,4,5}
\draw (\i*72+18:\rad) node[] (x\i) {};
\draw (0,0) node[draw=none] (origin) {};

\foreach \i/\j in {1/2, 2/3, 3/4, 4/5, 5/1}
{
\draw (x\i) -- (barycentric cs:x\i=5,x\j=5,origin=-.4) node[usStyle] {} -- (x\j);
\draw (x\i) -- (barycentric cs:x\i=5,x\j=5,origin=0) node[usStyle] (y\i) {} -- (x\j);
\draw (x\i) -- (barycentric cs:x\i=5,x\j=5,origin=.5) node[usStyle] {} -- (x\j);
}

\foreach \i/\j in {1/3, 2/4, 3/5, 4/1, 5/2}
{
\draw (x\i) -- (barycentric cs:x\i=3,x\j=3,origin=-.6) node[usStyle] {} -- (x\j);
\draw (x\i) -- (barycentric cs:x\i=3,x\j=3,origin=0) node[usStyle] (z\i) {} -- (x\j);
\draw (x\i) -- (barycentric cs:x\i=3,x\j=3,origin=.75) node[usStyle] {} -- (x\j);
}

\foreach \i/\j in {1/3, 2/4, 3/5, 4/1, 5/2}
{
\draw (y\j) edge[bend left=50] (y\i);
\draw (barycentric cs:y\i=3,y\j=3,origin=-3.53) node[usGStyle] {};
\draw (y\i) edge[bend left=12] (z\j);
}

\foreach \i/\j/\k in {1/3/5, 2/4/1, 3/5/2, 4/1/3, 5/2/4}
\draw (barycentric cs:y\i=3,z\j=3.35,y\k=.6) node[usGStyle] {};

\foreach \i/\j in {1/2, 2/3, 3/4, 4/5, 5/1}
\draw (z\j) edge[bend left=90, looseness=2] (z\i);

\foreach \i in {1,...,5}
\draw (barycentric cs:y\i=3,origin=2.0) node[usGStyle] {};

\end{tikzpicture}
\caption{A graph $G_D$ with maximum degree $D$.  In $G_D^2$, the black vertices form a clique of
order $5D/2$.  Each pair of black vertices with no common white neighbor 
has a gray neighbor that is adjacent only to them (though only a few such gray
vertices are shown).
\label{construction-fig}\label{fig1}}
\end{figure}

\begin{example}[\cite{HKP}]
\label{example1}
For every positive even integer $D$, there exists a 2-degenerate graph $G_D$
with maximum degree $D$ such that $\omega(G_D^2)=\frac52D$; see Figure~\ref{fig1}.
\smallskip

Fix a positive integer $D$ that is divisible by 4.  To form $G_D$, (1) start with 
the complete graph $K_5$.  (2) Replace each edge $vw$ with a copy of $K_{2,D/4}$,
identifying the vertices in the part of size 2 with $v$ and $w$.  (3) Now for each
pair, $x$ and $y$, of vertices of degree 2 with no common neighbor, add a vertex
$z_{xy}$ adjacent to both $x$ and $y$.  Call the resulting graph $G_D$.  
Clearly, $G_D$ is 2-degenerate.  It is easy to check that also
$\omega(G_D^2)=\frac52D$; the set $S$ of all vertices adjacent to two vertices
of the original $K_5$ is a clique in $G_D^2$.  Furthermore, $G_D$ has maximum
degree $4(D/4)=D$; this is the degree of each vertex of the original $K_5$ and
every other vertex has lower degree.  (If $D$ is even, but not divisible by 4,
perform the construction above with $D':=D+2$, but
for each edge on the outer 5-cycle of the original $K_5$, after (2) but before
(3), delete one vertex in the part of size $D'/4$.)
\exampleEnd
\end{example}

\begin{example}
\label{example2}
For each positive integer $D$, there exists a graph $H_D$ that is
2-degenerate and has maximum degree $D$, but such that $H_D^2$ is not
$(3D-5)$-degenerate; see Figure~\ref{example2-fig}.

Let $G_D$ be a $D$-regular graph with $4s$ vertices, for some $s\ge 3$ (for
example, $G_D$ could be a circulant); let $T:=V(G_D)$.  Subdivide
each edge of $G_D$; call the resulting graph $G_D'$ and call these new vertices $S$.
Build an auxiliary graph $J$ with vertex set $S$ and $vw\in E(J)$ if $v$ and $w$
are distinct and have no common neighbor in $G_D'$.  Note that $J$ is regular of degree
$|E(G_D)|-(2D-1)=2sD-(2D-1)>\frac12|E(G_D)|=\frac12|S|=\frac12|V(J)|$.  Thus, by
Dirac's Theorem, $J$ has a Hamiltonian cycle, $C$.  By deleting the edges of
$C$ and repeating the argument, we find $\lceil\frac12(D-2)\rceil$ edge-disjoint
Hamiltonian cycles $C_i$ in $J$.  If $D$ is even, let $J_1$ be the union of
these cycles $C_i$; if $D$ is odd, let $J_1$ be the union of all but one of
these cycles $C_i$ and a 1-factor from the final cycle.  Note that $J_1$ is
$(D-2)$-regular.  Now, for each edge $vw\in E(J_1)$, add to
$G'_D$ a new vertex $z_{vw}$ adjacent to $v$ and $w$; call the resulting graph
$H_D$.  Now we can check that $H_D$ is 2-degenerate with maximum
degree $D$ and $H_D^2[S]$ is regular of degree $3D-4$.  \exampleEnd
\end{example}

\begin{figure}[!h]
\centering
\begin{tikzpicture}[xscale=.3, yscale=.25] 
\tikzstyle{uStyle}=[shape = circle, minimum size = 5.5pt, inner sep = 0pt,
outer sep = 0pt, draw, fill=white, semithick]
\tikzstyle{lStyle}=[shape = circle, minimum size = 4.5pt, inner sep = 0pt,
outer sep = 0pt, draw, fill=none, draw=none]
\tikzstyle{usStyle}=[shape = circle, minimum size = 2.5pt, inner sep = 0pt,
outer sep = 0pt, draw, fill=black, semithick]
\tikzstyle{usGStyle}=[shape = circle, minimum size = 2.5pt, inner sep = 0pt,
outer sep = 0pt, draw, fill=gray!40!white, semithick]
\tikzset{every node/.style=uStyle}
\def\rad{4.5cm}

\draw (0,2) node (x1) {} -- (0,-2) node (x2) {};
\foreach \ang in {0, 10, 20, 30}
{
\draw (x1) --++ (\ang:\rad);
\draw (x2) --++ (-\ang:\rad);
}
\draw (2,-6) node[lStyle] {\footnotesize{$G_D$}};

\begin{scope}[xshift=3.5in]

\draw (0,2) node (x1) {} -- (0,-2) node (x2) {};
\draw (barycentric cs:x1=1,x2=1) node[usStyle] {};
\foreach \ang in {0, 10, 20, 30}
{
\draw (x1) --++ (\ang:.65*\rad) node[usStyle] {} --++ (\ang:.35*\rad);
\draw (x2) --++ (-\ang:.65*\rad) node[usStyle] {} --++ (-\ang:.35*\rad);

\draw (2,-6) node[lStyle] {\footnotesize{$G_D'$}};
}

\end{scope}

\begin{scope}[xshift=7in]
\draw (0,0) node[usStyle] (y1) {} --++ (1.5,3.5) node[usStyle] {};
\draw (y1) --++ (1.5,-3.5) node[usStyle] {};
\draw (1,0) node[usStyle] (y2) {} --++ (1.5,3.5) node[usStyle] {};
\draw (y2) --++ (1.5,-3.5) node[usStyle] {};

\draw (2,-5.5) node[lStyle] {\footnotesize{$C_1$, $C_2$}};
\draw (2,-6.5) node[lStyle] {\footnotesize{$\in J$}};
\end{scope}

\begin{scope}[xshift=9.75in]
\draw (0,0) node[usStyle] (y1) {} --++ (1.5,3.5) node[usStyle] {};
\draw (y1) --++ (1.5,-3.5) node[usStyle] {};
\draw (y1) --++ (3,3.5) node[usStyle] {};
\draw (2,-6.0) node[lStyle] {\footnotesize{$J_1$}};
\end{scope}

\begin{scope}[xshift=13in]

\draw (0,2) node (x1) {} -- (0,-2) node (x2) {};
\draw (barycentric cs:x1=1,x2=1) node[usStyle] (y1) {};
\foreach \ang in {0, 10, 20, 30}
{
\draw (x1) --++ (\ang:.65*\rad) node[usStyle] {} --++ (\ang:.35*\rad);
\draw (x2) --++ (-\ang:.65*\rad) node[usStyle] {} --++ (-\ang:.35*\rad);

\draw (7.5,3.5) node[usStyle] (yJ1) {};
\draw (9,3.5) node[usStyle] (yJ2) {};
\draw (7.5,-3.5) node[usStyle] (yJ3) {};

\draw (y1) edge [bend right=22.5] (yJ1);
\draw (y1) edge [bend right=22.5] (yJ2);
\draw (y1) edge [bend left=22.5] (yJ3);

\draw (5.9,2) node[usGStyle] {};
\draw (7.2,2) node[usGStyle] {};
\draw (5.9,-2) node[usGStyle] {};

\draw (4,-6) node[lStyle] {\footnotesize{$H_D$}};
}

\end{scope}

\end{tikzpicture}
\caption{The process of forming $H_D$ from $G_D$ in Example~\ref{example2}, when $D=5$, as viewed from the
perspective of a single edge in $G_D$. \label{example2-fig}}
\end{figure}
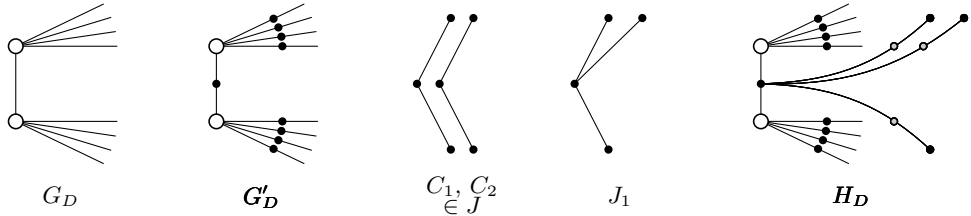

Example~\ref{example2} shows that even proving the bound $\chi(G^2)\le 3D-4$
will require new ideas.  The degeneracy of $G^2$ alone will not suffice.
%
(We can generalize Example~\ref{example2} to
something similar for $k$-degenerate graphs with arbitrary $k\ge 2$.
But to keep our focus here on 2-degenerate graphs, we defer these details to
Example~\ref{example3}.)

\subsection{Overview of Proofs}

The proofs of both Theorems~\ref{main1} and~\ref{main2} follow a similar
outline, though the latter is more complicated.  
We have made some effort to minimize the constant in Theorem~\ref{main1}, but
have not made much effort in this direction in Theorem~\ref{main2}, prefering to
present a simpler proof.

\begin{defn}
\label{def1}
A \EmphE{$k$-degeneracy order}{0mm} for a graph $G$ is an order of $V(G)$ such that
each vertex has at most $k$ neighbors later in the order.
A graph $G$ is \EmphE{nice}{4mm}, w.r.t. a clique $S$ in $G^2$, if (a) $S$ is a clique in
$G^2$, (b) $S$ is an independent set in $G$, and (c) $G$ has a $2$-degeneracy
order $\sigma$ such that all vertices of $S$ appear consecutively in $\sigma$.
The notion is inspired by the construction of Hocquard, Kim, and Pierron,
described in Example~\ref{example1}, which is indeed nice.
\end{defn}

In Section~\ref{nice-sec}, we consider only graphs $G$ that are nice
w.r.t.~a maximum clique $S$ in $G^2$.  

\begin{thm}
\label{main3}
If $G$ is nice w.r.t.~a maximum clique $S$ in $G^2$ and $\Delta(G)\le D$,
then $\omega(G^2)\le \frac52D$.  
\end{thm}

Theorem~\ref{main3} is exactly sharp, as
witnessed by the graphs in Example~\ref{example1}.  
(In fact, these graphs and a class of similar ones are the unique extremal
examples, as can be seen by studying the proof more carefully. We give more
details following the proof of Proposition~\ref{prop1}.) To prove
Theorem~\ref{main1}, we consider a graph $G$ that is 2-degenerate and fix a
vertex subset $S$ such that $S$ is a maximum clique in $G^2$.  
We show that there exists $S'\subseteq S$ such
that $|S\setminus S'|\le 72$ and there exists a graph $G'$ (a subgraph of $G$)
such that $G'$ is nice w.r.t.~$S'$.  Then the result follows by applying
Theorem~\ref{main3} to $G'$.  To prove Theorem~\ref{main2}, our approach is
similar.  We consider a graph $G$ with $\mad(G)<4$ and fix a vertex subset $S$
such that $S$ is a maximum clique in $G^2$.
We show that there exists $S'\subseteq S$ such
that $|S\setminus S'|\le 460$ and there exists a graph $G'$ (a subgraph of $G$)
such that $G'$ is nice w.r.t.~$S'$.  Then the result again follows by applying
Theorem~\ref{main3} to $G'$.  

In the proof of Theorem~\ref{main1} (resp.~Theorem~\ref{main2}), we introduce a technique
of ``token passing'', where vertices are deleted in a 2-degeneracy
(resp.~3-degeneracy) order
$\sigma$ and each vertex passes tokens, immediately before it is deleted, to
its neighbors later in $\sigma$.  These tokens facilitate a sort of amortized
analysis that is common in certain areas of algorithm analysis.  However, we
have not previously seen it used for the types of problems we consider in the
present paper.  And we believe it is likely to be applicable to further similar
problems.

\begin{remark}
After the preparation of this manuscript, Kim and Lian~\cite{KL} extended our ideas
to strengthen Theorem~\ref{main1}.  Specifically, they showed that if $G$ is 2-degenerate
and has maximum degree $D$ at least $3.3\times 10^6$, then $\omega(G^2)\le \frac52D$.
This upper bound on $\omega(G^2)$ is best possible.
\end{remark}

\section{Big Cliques in Squares of Nice 2-Degenerate Graphs}
\label{nice-sec}

The goal of this section is to prove the following result.

\begin{thm}
Let $H$ be a multigraph with $\Delta(H)\le D$, for some positive integer $D$.
If each edge of $H$ shares at least one endpoint with all but at most $D-2$
other edges of $H$, then $|E(H)|\le \frac52 D$.
\label{thm1}
\end{thm}

Intuitively, Theorem~\ref{thm1} essentially rephrases Theorem~\ref{main3} in
terms of edges in multigraphs, for which we have standard
terminology.  More formally, we have the following proposition.

\begin{prop}
Theorem~\ref{thm1} implies Theorem~\ref{main3}.
\label{prop1}
\end{prop}
\begin{proof}
Let $G$\aside{$G$, $S$} be nice w.r.t.~a maximum clique $S$ in $G^2$ and $\Delta(G)\le D$.  We
may assume that $V(G)\setminus S$ is an independent set; otherwise we delete all
edges induced by $V(G)\setminus S$ and again call the resulting graph $G$ (note
that it still satisfies all hypotheses of Theorem~\ref{main3}).  Let
\Emph{$\sigma$} be a 2-degeneracy order witnessing that $G$ is nice w.r.t.~$S$
and partition $V(G)$ into \Emph{$R,S,T$}, where $R$ are vertices before $S$ in
$\sigma$, and $T$ are vertices after $S$.  Form \Emph{$H$} from $G$ as follows:
delete all vertices of $R$ and contract one edge incident to each vertex of
$S$.  Figure~\ref{H-from-G-fig} shows an example.

\begin{figure}[!ht]
\centering
\begin{tikzpicture}[scale=.4]
\tikzstyle{uStyle}=[shape = circle, minimum size = 5.5pt, inner sep = 0pt,
outer sep = 0pt, draw, fill=white, semithick]
\tikzstyle{lStyle}=[shape = circle, minimum size = 4.5pt, inner sep = 0pt,
outer sep = 0pt, draw, fill=none, draw=none]
\tikzstyle{usStyle}=[shape = circle, minimum size = 2.5pt, inner sep = 0pt,
outer sep = 0pt, draw, fill=black, semithick]
\tikzstyle{usGStyle}=[shape = circle, minimum size = 2.5pt, inner sep = 0pt,
outer sep = 0pt, draw, fill=gray!40!white, semithick]
\tikzset{every node/.style=uStyle}
\def\rad{4.5cm}

\foreach \i in {1,2,3,4,5}
\draw (\i*72+18:\rad) node[] (x\i) {};
\draw (0,0) node[draw=none] (origin) {};

\foreach \i/\j in {1/2, 2/3, 3/4, 4/5, 5/1}
{
\draw (x\i) -- (barycentric cs:x\i=5,x\j=5,origin=-.4) node[usStyle] {} -- (x\j);
\draw (x\i) -- (barycentric cs:x\i=5,x\j=5,origin=0) node[usStyle] (y\i) {} -- (x\j);
\draw (x\i) -- (barycentric cs:x\i=5,x\j=5,origin=.5) node[usStyle] {} -- (x\j);
}

\foreach \i/\j in {1/3, 2/4, 3/5, 4/1, 5/2}
{
\draw (x\i) -- (barycentric cs:x\i=3,x\j=3,origin=-.6) node[usStyle] {} -- (x\j);
\draw (x\i) -- (barycentric cs:x\i=3,x\j=3,origin=0) node[usStyle] (z\i) {} -- (x\j);
\draw (x\i) -- (barycentric cs:x\i=3,x\j=3,origin=.75) node[usStyle] {} -- (x\j);
}

\foreach \i/\j in {1/3, 2/4, 3/5, 4/1, 5/2}
{
\draw (y\j) edge[bend left=50] (y\i);
\draw (barycentric cs:y\i=3,y\j=3,origin=-3.53) node[usGStyle] {};
\draw (y\i) edge[bend left=12] (z\j);
}

\foreach \i/\j/\k in {1/3/5, 2/4/1, 3/5/2, 4/1/3, 5/2/4}
\draw (barycentric cs:y\i=3,z\j=3.35,y\k=.6) node[usGStyle] {};

\foreach \i/\j in {1/2, 2/3, 3/4, 4/5, 5/1}
\draw (z\j) edge[bend left=90, looseness=2] (z\i);

\foreach \i in {1,...,5}
\draw (barycentric cs:y\i=3,origin=2.0) node[usGStyle] {};


\draw (7.0,-.8) node[lStyle] {\huge{$\rightarrow$}};

\begin{scope}[xshift=5.5in]

\foreach \i in {1,2,3,4,5}
\draw (\i*72+18:\rad) node[] (x\i) {};
\draw (0,0) node[draw=none] (origin) {};

\foreach \i/\j in {1/2, 2/3, 3/4, 4/5, 5/1}
{
\draw (x\j) edge[bend left=10] (x\i);
\draw (x\j) edge (x\i);
\draw (x\j) edge[bend right=10] (x\i);
}

\foreach \i/\j in {1/3, 2/4, 3/5, 4/1, 5/2}
{
\draw (x\j) edge[bend left=7] (x\i);
\draw (x\j) edge (x\i);
\draw (x\j) edge[bend right=7] (x\i);
}
\end{scope}

\end{tikzpicture}
\caption{The graph $G_D$ (left) from Example~\ref{example1} and the graph $H_D$
(right) produced from $G_D$ by the operations described in the first paragraph proving
Proposition~\ref{prop1}.
\label{H-from-G-fig}}
\end{figure}
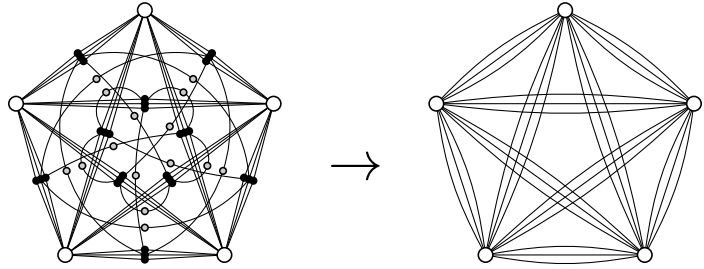

It is easy to check that
$\Delta(H)\le \Delta(G)$ and $|E(H)|=|S|$.  Consider $v\in S$.  Since $\sigma$
is a 2-degeneracy order, each neighbor $w\in R$ of $v$ gives rise to at most one
edge $vv'\in E(G^2)$.  If $v$ has at most one neighbor in $T$, then
$v$ has degree at most $2(D-1)$ in $G^2$, so $|S|\le 2D$.  So assume instead
that $v$ has at least two neighbors in $T$ (exactly two, in fact, since $\sigma$
is a 2-degeneracy order).  Since $\Delta(G)\le D$, vertex $v$ has at most $D-2$
neighbors in $R$.  So for all but at most $D-2$ vertices $v'$ in $S$, vertex $v$
has a neighbor $w$ in $T$ in common with $v'$.  Thus, the edge in $H$ arising
from $v$ has a common endpoint with all edges of $H$ except for at most $D-2$.
Hence, $H$ satisfies the hypotheses of Theorem~\ref{thm1}.  As a result,
$\omega(G^2)=|E(H)|\le \frac52D$, as desired.  
\end{proof}

To see that Theorem~\ref{thm1} is best possible, start with $K_5$ and replace
each edge with $D/4$ parallel edges.  We can also say a bit about uniqueness and
stability.  The bulk of the work in proving Theorem~\ref{thm1} is considering
the case $|V(H)|\ge 6$ and proving that $|E(H)|\le \frac52D$; in fact, we can
prove that this inequality is strict.  Thus, every graph $H$ that satisfies
Theorem~\ref{thm1} with equality has $|V(H)|=5$.  However, we do not have much
``stability'', since (for infinitely many values of $D$) there is a graph $H$
with $|V(H)|=6$ and $|E(H)|=\frac52D-5$.  Namely, start with $K_6$ and replace
each edge with $\frac16(D-2)$ parallel edges.  In fact, even the extremal examples on
5 vertices are not unique.  Such examples require only that the graph is
$D$-regular and that each edge shares an endpoint with all but at most $D-2$
other edges.  For large values of $D$, there are many such graphs.  For example,
starting with the example in the first sentence of this paragraph, we can pick
an arbitrary 5-cycle $\C$ and delete $2$ copies of each edge on $\C$.  This
gives an example with maximum degree $D-4$ that differs from that first example.

\subsection{Proof of Theorem~\ref{thm1}}

We assume Theorem~\ref{thm1} is false and let $H$ be a counterexample
minimizing $|E(H)|+|V(H)|$.  

\begin{defn}
For a multigraph $H$, we write $H_0$\aside{$H_0$, $\overline{H_0}$} for the
underlying simple graph and $\overline{H_0}$ for the complement of this simple graph.
For each $vw\in E(H)$, we write \Emph{$\mu(vw)$} to denote the multiplicity of
edge $vw$.  For each $v\in V(H)$, let \Emph{$E(v)$} denote the set of edges
incident to $v$.
\end{defn}

If we can show that $H_0$ is a complete graph, say $K_h$, then we can conclude
with a simple counting argument, taking advantage of the symmetry of $K_h$; see
Claim~\ref{lem8}.  So this is the main work of the present section.
Equivalently, we show that $\overline{H_0}$ is edgeless.  To this end, we show
that $\overline{H_0}$ is, first, a vertex disjoint union of stars
(Claim~\ref{lem5}); second, a matching (Claim~\ref{lem6}); and third, a graph
with at most one edge (Claim~\ref{lem7}).

\begin{clm}
$H$ is connected and $|V(H)|\ge 6$.
\label{lem0}
\end{clm}
\begin{clmproof}
Since $|E(H)|+|V(H)|$ is minimum, $H$ has no isolated vertices.  If $H$ is
disconnected, then some edge $e$ of $H$ is in a component with at most half of
the edges of $H$.  But now $e$ has no endpoint in common with at least
$\frac12|E(H)|>\frac54D$, a contradiction.  If $|V(H)|\le 5$, then $|E(H)| = 
\frac12 \sum_{v\in V(H)}d(v)\le \frac52D$, so $H$ is not a counterexample.
\end{clmproof}

\begin{clm}
If $vw\in E(H)$, then $d(v)+d(w)-\mu(vw) \ge \frac32D+2$.
\label{lem1}
\end{clm}
\begin{clmproof}
By definition, 
$D-2\ge 
|E(H)\setminus (E(v)\cup E(w))|=
|E(H)|-(d(v)+d(w)-\mu(vw))$.
Thus, $d(v)+d(w)-\mu(vw)\ge |E(H)|-(D-2)\ge \frac32D+2$.
\end{clmproof}

\begin{clm}
We have $\delta(H_0)\ge 3$.
\label{lem2}
\end{clm}
\begin{clmproof}
By Claim~\ref{lem0}, $d_{H_0}(v)\ge 1$ for each $v\in V(H_0)$.
Suppose there exists $v\in V(H_0)$ with $d_{H_0}(v)=1$; denote its neighbor by
$w$.  Claim~\ref{lem1} gives $\frac32 D < d(w)+(d(v)-\mu(vw))=d(w)\le D$, a
contradiction.  So suppose instead that $d_{H_0}(v)=2$; denote the neighbors of
$v$ by $w$ and $x$.  By symmetry, assume that $\mu(vw)\le \mu(vx)$, so
$\mu(vw)\le \frac12D$.  Similar to above, $\frac32 D <d(v)+d(x) - \mu(vx) = d(x)+ \mu(vw) \le \frac32D$, a contradiction.  Thus, $\delta(H_0)\ge 3$, as claimed.
\end{clmproof}

\begin{clm}
\label{lem3}
If $v,w,x,y\in V(H)$ are distinct vertices and $vw,xy\in E(H)$, then
$\mu(vx)+\mu(vy)+\mu(wx)+\mu(wy) \ge \frac12D+2$.
\end{clm}
\begin{clmproof}
If an edge $e$ shares an endpoint with both $vw$ and $xy$, then $e$ is counted
by $\mu(vx)+\mu(vy)+\mu(wx)+\mu(wy)$.
Each of $vw$ and $xy$ shares an endpoint with all but at most $D-2$ edges
of $H$, so $\mu(vx)+\mu(vy)+\mu(wx)+\mu(wy)\ge |E(H)|-|\{vw,xy\}|-2(D-2)\ge
\frac52 D-(2D-2)=\frac12D+2$.
\end{clmproof}

\begin{clm}
For all distinct $v,w\in V(H)$, we have 
$\sum_{x\in V\setminus\{v,w\}}\mu(vx)>\frac12 D$;
so $\mu(vw)<\frac12D$.
\label{lem4}
\end{clm}
\begin{clmproof}
First suppose that $vw\in E(H)$.
By Claim~\ref{lem1}, we have $\frac32D<d(w)+d(v)-\mu(vw) = 
d(w)+(\sum_{x\in V\setminus\{v\}}\mu(vx))-\mu(vw) =
d(w)+\sum_{x\in V\setminus\{v,w\}}\mu(vx) \le
D+\sum_{x\in V\setminus\{v,w\}}\mu(vx)$.  Thus, $\frac12D < 
\sum_{x\in V\setminus\{v,w\}}\mu(vx)$, as claimed. So $\mu(vw)<D-\frac12D$. 
Suppose instead that $vw\notin E(H)$.
By Claim~\ref{lem0}, $v$ has some incident edge $vz$.  As above,
$\sum_{x\in V\setminus\{v,z\}}\mu(vx)>\frac12D$.  Now $\mu(vw)=0$, so
$\sum_{x\in V\setminus\{v,w\}}\mu(vx)>\sum_{x\in V\setminus\{v,w,z\}}\mu(vx) =
\sum_{x\in V\setminus\{v,z\}}\mu(vx) > \frac12D$, as claimed.
\end{clmproof}

\begin{clm}
The graph $\overline{H_0}$ is a vertex disjoint union of stars. 
\label{lem5}
\end{clm}
\begin{clmproof}
First suppose, to the contrary, that $\overline{H_0}$ contains a triangle $vwx$;
see Figure~\ref{lem5-fig}(a).
That is, $vw,vx,wx\notin E(H_0)$.  Since $\delta(H_0)\ge 3$, there exist $y$ 
such that $vy\in E(H_0)$.  By Claim~\ref{lem4}, we have 
$\sum_{z\in V\setminus\{w,y\}}\mu(wz)>\frac12 D$
and $\sum_{z\in V\setminus\{x,y\}}\mu(xz)>\frac12 D$.  Since $\mu(wx)=0$, the
sets of edges counted by these two inequalities are disjoint.  Thus, there exist
more than $D$ edges that share no endpoint with edge $vy$.  This contradicts our
definition of $H$ (in the hypothesis of the theorem).  Thus, $\overline{H_0}$ is
triangle-free.

Now suppose instead that $\overline{H_0}$ contains edges $vw,wx,xy$; see
Figure~\ref{lem5-fig}(b).  
By the previous paragraph, $vx,wy\notin E(\overline{H_0})$.  
Now Claims~\ref{lem3} and~\ref{lem4} give $\frac12D+2\le
\mu(vw)+\mu(wx)+\mu(xy)+\mu(vy)=\mu(vy)<\frac12D$, a contradiction.
This finishes the proof of the claim.
\end{clmproof}

\begin{figure}[!h]
\centering
\begin{tikzpicture}[scale=1.1, yscale=.6, thick]
\tikzstyle{uStyle}=[shape = circle, minimum size = 5.5pt, inner sep = 0pt,
outer sep = 0pt, draw, fill=white, semithick]
\tikzstyle{lStyle}=[shape = circle, minimum size = 4.5pt, inner sep = 0pt,
outer sep = 0pt, draw, fill=none, draw=none]
\tikzset{every node/.style=uStyle}
\def\off{3.5mm}
\def\myRad{5.5mm}

\draw (0,0) node (v) {} ++(.3,-2) node (w) {} ++(1.4,0) node (x) {} ++(.3,2) node (y) {};

\draw[dash pattern=on 2pt off 2pt] (v) -- (w) -- (x) -- (v);
\draw (v) -- (y);

\draw (v) ++(-\off,0) node[lStyle] {\footnotesize{$v$}};
\draw (w) ++(-\off,0) node[lStyle] {\footnotesize{$w$}};
\draw (x) ++(\off,0) node[lStyle] {\footnotesize{$x$}};
\draw (y) ++(\off,0) node[lStyle] {\footnotesize{$y$}};

\draw (w) --++(245:\myRad) arc (245:295:\myRad) -- (w);
\draw (x) --++(245:\myRad) arc (245:295:\myRad) -- (x);

\draw (1,-3.5) node[lStyle] {\footnotesize{(a)}};

\begin{scope}[xshift=2.15in]
\draw (0,0) node (v) {} ++(.3,-2) node (w) {} ++(1.4,0) node (x) {} ++(.3,2) node (y) {};

\draw[dash pattern=on 2pt off 2pt] (v) -- (w) -- (x) -- (y);
\draw (x) -- (v) (y) -- (w);

\draw (v) ++(-\off,0) node[lStyle] {\footnotesize{$v$}};
\draw (w) ++(-\off,0) node[lStyle] {\footnotesize{$w$}};
\draw (x) ++(\off,0) node[lStyle] {\footnotesize{$x$}};
\draw (y) ++(\off,0) node[lStyle] {\footnotesize{$y$}};

\draw (w) --++(245:\myRad) arc (245:295:\myRad) -- (w);
\draw (x) --++(245:\myRad) arc (245:295:\myRad) -- (x);

\draw (1,-3.5) node[lStyle] {\footnotesize{(b)}};

\end{scope}

\end{tikzpicture}
\caption{The proof of Claim~\ref{lem5}: (a) a triangle $vwx$ in $\overline{H_0}$,
and (b) a $P_4$ $vwxy$ in $\overline{H_0}$.\label{lem5-fig}}
\end{figure}
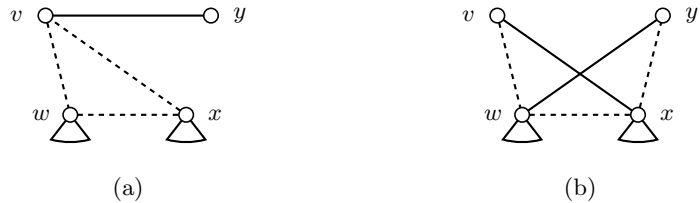

\begin{clm}
The graph $\overline{H_0}$ is a matching.
\label{lem6}
\end{clm}
\begin{clmproof}
Suppose, to the contrary, that there exist distinct $v,w,x\in V(H_0)$
with $vw,vx\in E(\overline{H_0})$; that is, $vw,vx\notin E(H_0)$.  
Since $\delta(H_0)\ge 3$, there exists $y$ such that $vy\in
E(H_0)$.  By Claim~\ref{lem5}, we know $wx, wy, xy\in E(H_0)$;  
see Figure~\ref{lem67-fig}(a). 
Thus, Claim~\ref{lem3} gives $\mu(vw)+\mu(vx)+\mu(yw)+\mu(yx)>\frac12D$.
However, $\mu(vw)=\mu(vx)=0$, so $\mu(yw)+\mu(yx)>\frac12D$.
Since, $\delta(H_0)\ge 3$, there exist $y',z$ such that $y,y',z$ are distinct
neighbors of $v$ in $H_0$.  Repeating the arguments above for $y'$, we get that
$\mu(y'w)+\mu(y'x)>\frac12D$.  But now $vz$ has more than $D$ edges with which
it shares no endpoint (those counted by $\mu(yw)+\mu(yx)+\mu(y'w)+\mu(y'x)$),
a contradiction.
\end{clmproof}

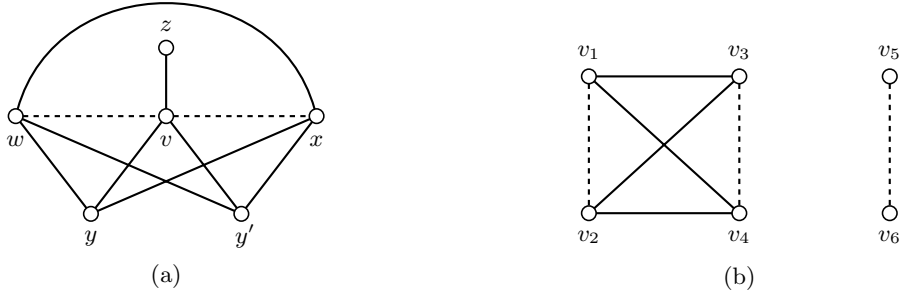
\begin{figure}[!h]
\centering
\begin{tikzpicture}[thick, yscale=1.3]
\tikzstyle{uStyle}=[shape = circle, minimum size = 5.5pt, inner sep = 0pt,
outer sep = 0pt, draw, fill=white, semithick]
\tikzstyle{lStyle}=[shape = circle, minimum size = 4.5pt, inner sep = 0pt,
outer sep = 0pt, draw, fill=none, draw=none]
\tikzset{every node/.style=uStyle}
\def\off{2.5mm}
\def\myRad{5.5mm}

\draw 
(0,0) node (v) {} 
(-2,0) node (w) {}
(2,0) node (x) {}
(-1,-1) node (y) {}
(1,-1) node (y') {}
(0,.7) node (z) {};

\draw (w) edge[bend left=70] (x);

\foreach \x in {v,w,x,y}
\draw (\x) ++ (0,-\off) node[lStyle] {\footnotesize{$\x$}};

\draw (y') ++ (0.5mm, -.9*\off) node[lStyle] {\footnotesize{$y'$}};
\draw (z) ++ (0, 0.9*\off) node[lStyle] {\footnotesize{$z$}};

\draw[dash pattern=on 2pt off 2pt] (w) -- (v) -- (x);

\draw (y) -- (v) -- (y') -- (x) -- (y) -- (w) -- (y') (v) -- (z);

\draw (0,-1.635) node[lStyle] {\footnotesize{(a)}};

\begin{scope}[xshift=3in, yscale=.7, yshift=-.165in]
\def\off{3.5mm}
\foreach \i/\x/\y in {1/-2/1, 2/-2/-1, 3/0/1, 4/0/-1, 5/2/1, 6/2/-1}
{
\draw (\x,\y) node (v\i) {};
\draw (v\i) ++ (0,\y*\off) node[lStyle] {\footnotesize{$v_{\i}$}};
}

\draw[dash pattern=on 2pt off 2pt] (v1) -- (v2) (v3) -- (v4) (v5) -- (v6);
\draw (v1) -- (v4) -- (v2) -- (v3) -- (v1);

\draw (0,-1.95) node[lStyle] {\footnotesize{(b)}};
\end{scope}
\end{tikzpicture}
\caption{The proofs of (a) Claim~\ref{lem6} and (b)
Claim~\ref{lem7}.\label{lem67-fig}}
\end{figure}

\begin{clm}
The graph $\overline{H_0}$ contains at most one edge.
\label{lem7}
\end{clm}
\begin{clmproof}
First suppose instead that $\overline{H_0}$ has at least 3 edges;
that is, by Claim~\ref{lem6}, there are $v_1v_2, v_3v_4, v_5v_6 \notin E(H_0)$.
See Figure~\ref{lem67-fig}(b).
By Claim~\ref{lem6}, we know $v_1v_3, v_2v_4\in E(H_0)$.
So Claim~\ref{lem3} implies that $\mu(v_2v_3)+\mu(v_1v_4)>\frac12D$.
Similarly, $\mu(v_1v_3)+\mu(v_2v_4)>\frac12D$.  We can now repeat both of these
arguments twice, first with $v_5v_6$ in place of $v_3v_4$, and second with
$v_5v_6$ in place of $v_1v_2$.  Thus, we have partitioned the 12 edges induced
(in $H_0$) by $v_1,\ldots,v_6$ into 6 pairs, each of which has sum of
multiplicities greater than $\frac12D$.  Thus, $|E(H)|>6(\frac12D)=3D$. 
However, applying the hypothesis of the theorem to edge $v_1v_3$ gives
$|E(H)|\le d(v_1)+d(v_3)+(D-2)<3D$, a contradiction.

Now suppose instead that $\overline{H_0}$ has only two edges.  That is,
$v_1v_2,v_3v_4\notin E(H_0)$.  As above, we conclude
that $\{v_1,v_2,v_3,v_4\}$ induces (in $H$) more than $D$ edges.
Now there exists $v_5v_6\in E(H_0)$ (with $v_5,v_6\notin \{v_1,v_2,v_3,v_4\}$;
recall that $|V(H)|\ge 6$, by Claim~\ref{lem0}).  But now there exist more than
$D$ edges with which $v_5v_6$ shares no endpoint, a contradiction.
\end{clmproof}

\begin{clm}
The graph $\overline{H_0}$ contains exactly one edge. Further, $|V(H_0)|=6$.
\label{lem8}
\end{clm}
\begin{clmproof}
Suppose to the contrary, by Claim~\ref{lem7}, that $\overline{H_0}$ contains no
edges; that is $H_0$ is a complete graph.  Let \Emph{$h$}$:=|V(H_0)|$.  
By assumption, every edge $vw\in E(H)$ shares an endpoint with all but at most
$D-2$ edges.  Thus, for every pair $v,w\in V(H)$, 
we have $d_H(v)+d_H(w)-\mu(vw)
\ge |E(H)|-(D-2)$.  Summing over all $\binom{h}2$ edges of $H_0$ gives
$$
\sum_{vw\in E(H_0)}\left(d_H(v)+d_H(w)-\mu(vw)\right)\ge \sum_{vw\in
E(H_0)}\left(|E(H)|-(D-2)\right)
$$
Each edge of $H$ is counted on the left exactly $2h-3$ times.  Thus, we get
$(2h-3)|E(H)| \ge \binom{h}2|E(H)|-\binom{h}2(D-2)$.
This simplifies to $|E(H)|\le (D-2)\frac{h(h-1)}{(h-2)(h-3)} =
(D-2)(1+\frac{4h-6}{(h-2)(h-3)}) = (D-2)(1+\frac{6}{h-3}+\frac{-2}{h-2})$.
This expression is decreasing with $h$, so is maximized when $h=6$, where we get
$|E(H)|\le (D-2)(1+\frac{6}3+\frac{-2}4) = (D-2)\frac52$.  Thus,
$\overline{H_0}$ contains exactly one edge, as claimed.

Let $vw$ be the unique edge in $\overline{H_0}$.
Now we repeat the argument above.  The main difference is that 
the sum on the right now has only $\binom{h}2-1$ terms.
There are also some edges that are counted only $2h-4$ times, but we can still
upper bound the left side by $(2h-3)|E(H)|$.  That is $(2h-3)|E(H)|\ge
(\binom{h}2-1)|E(H)|-(\binom{h}2-1)(D-2)$.  This simplifies to $|E(H)|\le
(D-2)\frac{h^2-h-2}{(h-1)(h-4)} = (D-2)(1+\frac{4h-6}{(h-1)(h-4)}) =
(D-2)(1+\frac{2/3}{h-1}+\frac{10/3}{h-4})$.  Again this expression is decreasing
with $h$.  When $h=6$, we get $|E(H)|\le \frac{14}5(D-2)$, so we cannot rule out
that case.  But when $h=7$, we get $|E(H)|\le \frac{20}9(D-2)$.  Thus $h=6$, as claimed.
\end{clmproof}

\begin{figure}[!h]
\centering
\begin{tikzpicture}[thick, scale=1.35]
\tikzstyle{uStyle}=[shape = circle, minimum size = 5.5pt, inner sep = 0pt,
outer sep = 0pt, draw, fill=white, semithick]
\tikzstyle{lStyle}=[shape = circle, minimum size = 4.5pt, inner sep = 0pt,
outer sep = 0pt, draw, fill=none, draw=none]
\tikzset{every node/.style=uStyle}
\def\off{3.5mm}
\def\myRad{5.5mm}

\foreach \i in {0,...,5}
\draw (60*\i:1cm) node (v\i) {};

\foreach \i/\j in {0/1, 0/2, 0/3, 0/4, 0/5, 1/3, 1/4, 1/5, 2/3, 2/4, 2/5, 3/4,
3/5, 4/5}
\draw (v\i) -- (v\j);

\draw (v2) ++ (-\off,0) node[lStyle] {\footnotesize{$x$}};
\draw (v1) ++ (\off,0) node[lStyle] {\footnotesize{$y$}};
\draw (0,-1.25) node[lStyle] {\footnotesize{$H_0$}};

\begin{scope}[xshift=1.6in]
\foreach \i in {0,...,5}
\draw (60*\i:1cm) node (v\i) {};

\foreach \i/\j in {0/3, 0/4, 0/5, 3/4, 3/5, 4/5}
\draw (v\i) edge[semithick] (v\j);

\foreach \i/\j in {0/1, 0/2, 1/3, 1/4, 1/5, 2/3, 2/4, 2/5}
{
\draw (v\i) edge[thin, bend left=5] (v\j);
\draw (v\j) edge[thin, bend left=5] (v\i);
}

\draw (v2) ++ (-\off,0) node[lStyle] {\footnotesize{$x$}};
\draw (v1) ++ (\off,0) node[lStyle] {\footnotesize{$y$}};
\draw (0,-1.25) node[lStyle] {\footnotesize{$H_1$}};
\end{scope}

\end{tikzpicture}
\caption{The graphs $H_0$ and $H_1$ in the proof of
Claim~\ref{lem9}.\label{lem9-fig}}
\end{figure}
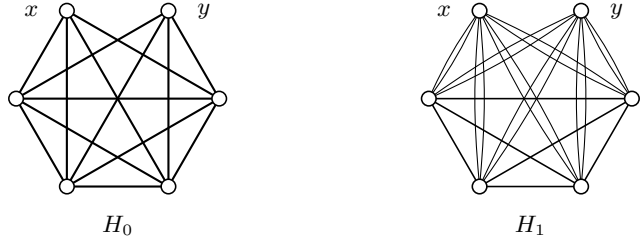

\begin{clm}
The graph $H$ does not exist; that is, Theorem~\ref{thm1} is true.
\label{lem9}
\end{clm}
\begin{clmproof}
Suppose the claim is false.  By Claim~\ref{lem8}, we must have $H_0 = K_6-e$.
Denote by $x,y$ the nonadjacent pair in $H_0$; see Figure~\ref{lem9-fig}.  We
define a multigraph $H_1$ that has $H_0$ as its underlying simple graph, has
$\mu(e)=2$ for each edge $e$ of $H_0$ incident to $x$ or $y$, and $\mu(e)=1$
for each other edge $e$ of $H_0$.  Similar to above, for each edge $vw\in
E(H_0)$ we have $d_H(v)+d_H(w)-\mu(vw)
\ge |E(H)|-(D-2)$.  Summing over all $14+8=22$ edges of $H_1$ gives\footnote{In
other words, $H_1$ is simply a convenient way to encode the multiplicities with
which we want to sum the 14 inequalities $d_H(v)+d_H(w)-\mu(vw)\ge |E(H)|-(D-2)$
arising from the 14 edges $vw$ of $H_0$.}
$$
\sum_{vw\in E(H_1)}\left(d_H(v)+d_H(w)-\mu(vw)\right)\ge \sum_{vw\in
E(H_1)}\left(|E(H)|-(D-2)\right)
$$
It is straightforward to check that each edge of $H$ is counted exactly 13 times.
Thus, the inequality above can be rewritten as $13|E(H)|\ge 22|E(H)|-22(D-2)$,
which simplifies to $|E(H)|\le \frac{22}9(D-2)$.  This contradiction proves the
claim, which in turn proves the theorem.
\end{clmproof}

\section{Big Cliques in Squares of 2-Degenerate Graphs}
\label{sec3}
Fix a positive integer $D$.  Let \Emph{$f(D)$} be the maximum size of a clique $S$ in
the square of a 2-degenerate graph $G$ with $\Delta(G)\le D$.
It is easy to verify that if $G$ is $k$-degenerate with $\Delta(G)\le D$, then
$G^2$ has degeneracy at most 
$k(D-1)+(D-k)(k-1)$; this is witnessed by any vertex
order showing that $G$ is $k$-degenerate.  So, when $k=2$, $G^2$ is
$(3D-4)$-degenerate.  Thus, $f(D)\le 3D-3$.  Example~\ref{example1} shows that $f(D)\ge
\frac52D$, when $D$ is even.
In fact, the graph $G$ in Example~\ref{example1} has a maximum clique $S$ in
$G^2$ the vertices of which form an independent set in $G$.  Furthermore, there
exists a 2-degeneracy order $\sigma$ for $G$ in which the vertices of $S$
appear consecutively.  The goal of this section is to show that if
$f(D)>\frac52D$, then this is witnessed by a graph $G$ where $S$ behaves
``nearly'' as it does above.  More specifically, we show the following
(recall the meaning of ``nice'' from Definition~\ref{def1}).

\begin{thm}
There exists a constant $c$ such that for every positive integer $D$ 
some 2-degenerate graph $G$ with $\Delta(G)\le D$ is nice
w.r.t.~a clique $S$ in $G^2$ and $|S|\ge f(D)-c$.  In fact, $c=72$ suffices and,
when $D$ is sufficiently large, $c=60$ suffices.
\end{thm}

\begin{proof}
We prove the theorem with $c=72$, but we have not made
significant effort to minimize $c$.  (A small trick at the end improves this to
$c=60$ when $D\ge 1729$.)

Fix a positive integer $D$\aside{$D$, $G$, $S$}, and let $G$ be a 2-degenerate
graph with maximum degree at most $D$ such that $G^2$ has a clique $S$ of
order $f(D)$.  Subject to this,
choose $G$ to minimize $|V(G)|+|E(G)|$.  We assume that $f(D) \ge \frac52D+60$;
otherwise, we are done by Example~\ref{example1}.
Note that every vertex $v\in V(G)$ must have a neighbor in $S$;
otherwise, deleting $v$ contradicts the minimality of $G$.

Let \Emph{$\sigma$} be a vertex order witnessing that $G$ is
2-degenerate.  Subject to this, choose $\sigma$ so the first vertex in $S$
appears as late as possible in $\sigma$.
We delete the vertices of $G$ in the order $\sigma$.  Each time
we delete a vertex $v$ of $S$, vertex $v$ gives a ``primary'' token to each
of its neighbors later in $\sigma$.  Moreover, each time we delete a
vertex $v$ that currently holds $s$ primary tokens (for some $s\ge 1$), $v$
gives to each of its neighbors later in $\sigma$ exactly $s$ ``secondary''
tokens; see Figure~\ref{fig-tokens}.

\begin{figure}[!h]
\centering
\begin{tikzpicture}[thick, scale=.75]
\tikzstyle{uStyle}=[shape = circle, minimum size = 10.5pt, inner sep = 0pt,
outer sep = 0pt, draw, fill=white, semithick]
\tikzstyle{lStyle}=[shape = circle, minimum size = 4.5pt, inner sep = 0pt,
outer sep = 0pt, draw, fill=none, draw=none]
\tikzset{every node/.style=uStyle}
\def\off{5mm}
\def\myRad{6.5mm}

\foreach \ang/\rad/\name in {90/2/2, 150/1/4, 30/1/5, 210/2/3, 270/1/6,
330/2/7, 30/2.5/1}
\draw (\ang:\rad*\myRad) node (v\name) {\footnotesize{\name}};

\foreach \x/\y in {1/2, 1/7, 2/4, 2/5, 3/4, 3/6, 4/5, 4/6, 5/6, 5/7, 6/7}
\draw (v\x) -- (v\y);

\begin{scope}[xshift=2in]
\foreach \ang/\rad/\name in {90/2/2, 150/1/4, 30/1/5, 210/2/3, 270/1/6,
330/2/7}
\draw (\ang:\rad*\myRad) node (v\name) {\footnotesize{\name}};

\foreach \x/\y in {2/4, 2/5, 3/4, 3/6, 4/5, 4/6, 5/6, 5/7, 6/7}
\draw (v\x) -- (v\y);

\foreach \tok/\x/\ang in {
{{\bolder 1}}/2/90,
{{\bolder 1}}/7/270
}
\draw (v\x) ++ (\ang:\off) node[lStyle] {\footnotesize{\tok}};
\end{scope}

\begin{scope}[xshift=4in]
\foreach \ang/\rad/\name in {150/1/4, 30/1/5, 210/2/3, 270/1/6,
330/2/7}
\draw (\ang:\rad*\myRad) node (v\name) {\footnotesize{\name}};

\foreach \x/\y in {3/4, 3/6, 4/5, 4/6, 5/6, 5/7, 6/7}
\draw (v\x) -- (v\y);

\foreach \tok/\x/\ang in {
{1 {\bolder 2}}/4/90,
{1 {\bolder 2}}/5/90,
{{\bolder 1}}/7/270
}
\draw (v\x) ++ (\ang:\off) node[lStyle] {\footnotesize{\tok}};
\end{scope}

\draw (10.2,-1.625) node[lStyle] {\Large{$\downarrow$}};

\begin{scope}[yshift=-.3in]
\begin{scope}[xshift=4in, yshift=-1in]
\foreach \ang/\rad/\name in {150/1/4, 30/1/5, 270/1/6,
330/2/7}
\draw (\ang:\rad*\myRad) node (v\name) {\footnotesize{\name}};

\foreach \x/\y in {4/5, 4/6, 5/6, 5/7, 6/7}
\draw (v\x) -- (v\y);

\foreach \tok/\x/\ang in {
{1 {\bolder {2 3}}}/4/90,
{1 {\bolder 2}}/5/90,
{{\bolder 3}}/6/270,
{{\bolder 1}}/7/270
}
\draw (v\x) ++ (\ang:\off) node[lStyle] {\footnotesize{\tok}};
\end{scope}

\begin{scope}[xshift=2.25in, yshift=-1in]
\foreach \ang/\rad/\name in {30/1/5, 270/1/6, 330/2/7}
\draw (\ang:\rad*\myRad) node (v\name) {\footnotesize{\name}};

\foreach \x/\y in {5/6, 5/7, 6/7}
\draw (v\x) -- (v\y);

\foreach \tok/\x/\ang in {
{1 {\bolder 2} 3 {\bolder 4}}/5/90,
{2 {\bolder{3 4}}}/6/270,
{{\bolder 1}}/7/270
}
\draw (v\x) ++ (\ang:\off) node[lStyle] {\footnotesize{\tok}};
\end{scope}

\begin{scope}[xshift=.7in, yshift=-1in]
\foreach \ang/\rad/\name in {270/1/6, 330/2/7}
\draw (\ang:\rad*\myRad) node (v\name) {\footnotesize{\name}};

\foreach \x/\y in {6/7}
\draw (v\x) -- (v\y);

\foreach \tok/\x/\ang in {
{2 {\bolder{3 4 5}}~~~~}/6/270,
{~~~~{\bolder 1} 2 4 \bolder 5}/7/270
}
\draw (v\x) ++ (\ang:\off) node[lStyle] {\footnotesize{\tok}};
\end{scope}

\begin{scope}[xshift=-.883in, yshift=-1in]
\foreach \ang/\rad/\name in {330/2/7}
\draw (\ang:\rad*\myRad) node (v\name) {\footnotesize{\name}};

\foreach \tok/\x/\ang in {
{{\bolder 1} 2 3 4 \bolder{5 6}}/7/270
}
\draw (v\x) ++ (\ang:\off) node[lStyle] {\footnotesize{\tok}};
\end{scope}

\foreach \x/\y in {0.5/-2.7, 4.35/-2.7, 8.4/-2.7}
\draw (\x,\y) node[lStyle] {\Large{$\leftarrow$}};

\end{scope}

\foreach \x/\y in {2.85/.34, 7.6/.34}
\draw (\x,\y) node[lStyle] {\Large{$\rightarrow$}};

\end{tikzpicture}
\caption{A portion of a 2-degenerate graph $G$ and the positions of its tokens as
its vertices are deleted in a 2-degeneracy order; here
$S=\{1,\ldots,7\}$.  The number of each token indicates the vertex where
the token originated.  Primary tokens are shown in
bold. At each vertex, we show at most one token received from each other
vertex, giving preference to primary tokens.\label{fig-tokens}}
\end{figure}
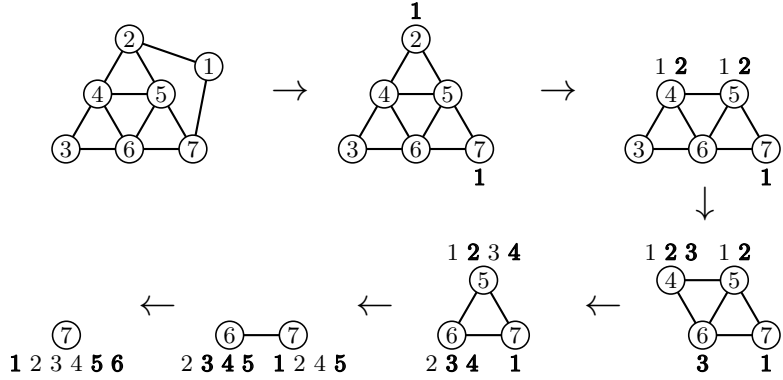

We will analyze this process to show that all but a constant number of vertices
in $S$ induce an independent set.  That is, there exists $\widehat{S}\subseteq
S$, of bounded size, such that $S\setminus \widehat{S}$ is an independent set.
Furthermore, we can modify $G$, $S$, and $\sigma$ to get a graph
$G'$ that is 2-degenerate, as witnessed by a vertex order $\sigma'$, and $G'$ is
nice w.r.t.~a clique $S'$ in $(G')^2$.  Finally, $|S|-|S'|\le c$. 
Let $\tok(v)$\aside{$\tok(v)$} (resp.~$\prim(v)$\aaside{$\prim(v)$}{4mm}) denote the number of
tokens (resp.~primary tokens) that each vertex $v$ holds immediately before it 
is deleted.

\begin{clm}
\label{clm0}
If $v\in S$ and $w_1,w_2$ are the neighbors of $v$ later in $\sigma$, if they
exist, then
\begin{equation}
1+\tok(v)+D+\prim(w_1)+\prim(w_2)+6\ge |S|\ge \frac52D+60.\tag{$\star$}
\label{key-ineq}
\end{equation}
\end{clm}

\begin{clmproof}
Fix $v\in S$.  
Recall that $v$ must be adjacent in $G^2$ to every
vertex in $S$.  Each neighbor $w$ of $v$ that precedes $v$ in $\sigma$ can be
the source of at most one such adjacency in $G^2$ to a vertex $x$ from which $v$
receives no token (neither primary nor secondary), so in total $v$ can have at
most $D$ such adjacencies.  Each other adjacency in $G^2$ must be accounted for
either (a) by a token received by $v$ or (b) by a primary token that has been
or will be received by $w_1$ or $w_2$ or (c) by being an adjacency to $w_1$ or
$w_2$ or to a neighbor of some $w_i$ that comes later than $w_i$ in $\sigma$. 
At most $\tok(v)$ vertices of $S$ are handled by (a), and at most 6 vertices of
$S$ are handled by (c).  The number handled by (b) is at most
$\prim(w_1)+\prim(w_2)$.  This implies~\eqref{key-ineq}.
\end{clmproof}
\smallskip

To analyze the movement of tokens, we use the following definitions (all
w.r.t.~$G$, $S$, and $\sigma$).
Let $\EmphE{\Big}{-4mm}:=\{v\in V(G): \prim(v) > \frac14D\}$.
Let $\Emph{\Basic}:=\{v\in S: \tok(v) < \frac14D\}$.
Let $\EmphE{\NonBasic}{0mm}:=S\setminus\Basic$.
So $\Big\cap\Basic=\emptyset$.  And possibly both 
$\Big\not\subseteq\NonBasic$ and $\NonBasic\not\subseteq\Big$.
Note that each vertex $v$ gives primary tokens to at most two vertices, and each
of those vertices gives secondary tokens (due to $v$) to at most two more
vertices.  Thus, $v$ contributes at most 6 to the total number of tokens held by
vertices at any point.  So $\sum_{v\in V(G)}\tok(v)\le 6|S|\le 6f(D) \le 18 D$.  
Thus, we get that $|\NonBasic\cup \Big|\le 18D/(D/4)=72$.
We will later refine this argument to get a stronger upper bound.

\begin{clm}
\label{clm1}
Each vertex $v\in \Basic$ has two neighbors later in $\sigma$ that are both in $\Big$.
\end{clm}

\begin{clmproof}
Let $w_1,w_2$ be the neighbors of $v$ later in $\sigma$, if they exist.
Inequality~\eqref{key-ineq} simplifies to $\tok(v)+\prim(w_1)+\prim(w_2)\ge
\frac32D+53$.
Since $v$ is basic, we have $\tok(v)<\frac14D$.  So $\prim(w_1)+\prim(w_2)\ge
\frac54D+53$.  Clearly, $\prim(w_i)\le D$ for each $i\in\{1,2\}$, which implies that
also $\prim(w_i)\ge \frac14D+53$, for each $i\in\{1,2\}$.  So $w_1,w_2\in \Big$.
\end{clmproof}
\smallskip

Let $\Emph{W}:=\{w: w\in N(v)$ for some $v\in \NonBasic$ and $w$ appears later
in $\sigma$ than $v$\}.  
\begin{clm}
\label{clm1.5}
If some vertex $w$ comes after the first vertex in $S$ and $w\notin \Big$, then
$w\in W$.
\end{clm}
\begin{clmproof}
Since $G$ minimizes $|V(G)|+|E(G)|$, every edge has an endpoint in $S$.  Since
$\sigma$ puts the first vertex $v$ of $S$ as late as possible, every vertex $w$ after
$v$ is a later neighbor of some vertex $x\in S$.  
Every vertex in $\Basic$ has as its later neighbors in $\sigma$ exactly two
vertices in $\Big$, by Claim~\ref{clm1}.  
Thus, if $w\notin \Big$, then $x\in\NonBasic$; hence, by definition, $w\in W$.
\end{clmproof}

Now we modify $G$, $S$, $\sigma$ to construct \Emph{$G'$, $S'$, $\sigma'$}.
\begin{enumerate}
\item[(1)] Let $S':=\Basic\setminus W$.
\item[(2)] Move $\Big$ to the end of $\sigma$ (after the final vertex of $S'$) and
move $(\NonBasic\cup W)\setminus \Big$
to the start of $\sigma$ (before the first vertex of $S'$); call the resulting order $\sigma'$.
\item[(3)] Delete every edge of $G$ with both endpoints outside $S'$; call the
resulting graph $G'$.
\end{enumerate}
Now we check that $G'$, $S'$, $\sigma'$ satisfy the desired criteria.

\begin{clm}
\label{clm2}
$S'$ appears consecutively in $\sigma'$, and $S'$ and $V(G')\setminus S'$ are
each independent in $G'$. 
\end{clm}
\begin{clmproof}
By Claim~\ref{clm1}, every vertex of $S'$ has two later
neighbors in $\Big$.  And every vertex $v$ of $\Big$ has no later neighbors.  
(Such a neighbor $w$ would have $\tok(w)>\frac14D$, so $w$ would be excluded from
$\Basic$, and hence from $S'$.  Thus, edge $vw$ would be removed in (3).)
So
$\Big$ comes, in some order, at the end of $\sigma'$.  And every vertex 
not in $S'\cup\Big$ has no earlier neighbors in $S'$, so comes before the
first vertex of $S'$ in $\sigma'$ (by Claim~\ref{clm1.5} and (2) above); thus,
$S'$ is consecutive in $\sigma'$.  To see that $S'$ is an independent set in
$G'$, note that every vertex in $\Basic$ has two later neighbors in $\Big$. 
But $S'\subseteq \Basic$, so $S'\cap\Big\subseteq\Basic\cap\Big=\emptyset$. 
Thus, $S'$ is independent.  Finally, $V(G')\setminus S'$ is independent in $G'$
by step (3) above constructing $G'$.
\end{clmproof}

\begin{clm}
\label{clm3}
$G'$ is $2$-degenerate, as witnessed by $\sigma'$. 
\end{clm}
\begin{clmproof}
Consider steps (2) and (3) above.  Suppose we move some $v\in \Big$ to the end of
$\sigma$.  If $v$ has a neighbor $w$ later (in $\sigma$), then $w \notin
\Basic$, so $w\notin S'$.  Thus, in step (3) we delete edge $vw$.
Suppose instead that we move a vertex $v\in (\NonBasic\cup W)\setminus \Big$ to
the start of $\sigma$.  If $v$ has a neighbor $w$ earlier (in $\sigma$), then
$w\notin\Basic$, by Claim~\ref{clm2}; so $w\notin S'$.  Thus, again in step (3)
we delete $vw$.  So $\sigma'$ is a 2-degeneracy order for $G'$ because $\sigma$
is a 2-degeneracy order for $G$.
\end{clmproof}

\begin{clm}
\label{clm4}
$S'$ forms a clique in $G'^2$. 
\end{clm}
\begin{clmproof}
Every edge deleted when forming $G'$ has both endpoints outside
of $S'$.  But deleting such edges cannot impact adjacency in $G'^2$ between two
vertices in $S'$.  Thus, $G'^2[S'] = G^2[S']$.
\end{clmproof}

\begin{clm}
\label{clm5}
$|S\setminus S'|\le 72$.
More strongly (for $D\ge 144$), we have $|S\setminus S'|\le 60+12\times(144/D)$.
\end{clm}

\begin{clmproof}
Let \Emph{$X$}$:=\Big\cup\NonBasic\cup W$ and note from (1) that $|S\setminus
S'|\le |X|$.
To bound the size of $X$,  we use a discharging argument.  Every vertex $x\in
X$ begins with charge $\tok(x)$.  We consider each $x\in \Big\cup \NonBasic$
(note that this union does not include $W$) and 
redistribute charge so that $x$, $w_1$, and $w_2$ all finish with charge at
least $\frac14D$; again, $w_1$ and $w_2$ are the neighbors of $x$ that follow $x$ in
$\sigma$ (if they exist).  Recall that $\sum_{v\in V(G)}\tok(v)\le 6|S|\le 18D$.  
Thus, this proves that $|S\setminus S'|\le |X|\le 6|S|/(D/4) \le 18D/(D/4)=72$.

Fix $x\in \Big$.  By definition, $\prim(x)>\frac14D$.  Further, $\tok(w_i)\ge
\prim(x)>\frac14D$ for each $i\in\{1,2\}$.
So $x$ does not need to give charge to any neighbor later in
$\sigma$.  Thus, $x$ (and each of its later neighbors) finishes with charge at
least $\frac14D$.

Fix $x\in (\NonBasic\setminus\Big)$.  If $\tok(x)\ge \frac34D$, then $x$ gives
charge $\frac14D$ to each of its at most two neighbors that follows it in $\sigma$.
Suppose instead that $\frac12D\le \tok(x)<\frac34D$.
By~\eqref{key-ineq}, we have $\prim(w_1)+\prim(w_2)>\frac34D$. 
So clearly
$\prim(w_i)>\frac14D$ for some $i\in\{1,2\}$.  Thus, at least one later neighbor of
$x$ is big.  So $x$ gives charge $\frac14D$ to at most one neighbor later in
$\sigma$.  Thus, $x$ (and each of its later neighbors) finishes with charge at
least $\frac14D$.  

Finally, suppose that $\frac14D\le \tok(x)<\frac12D$.  So \eqref{key-ineq} gives $\prim(w_1)+\prim(w_2)>D$.
Thus, $w_i\in\Big$ for at least one $i\in\{1,2\}$, so $w_i$ needs no charge from $x$.
If $w_1,w_2\in \Big$, then neither $w_i$ receives charge from $x$,
so $x$ (and each $w_i$) finishes with charge at least $\frac14D$.
Suppose instead, by symmetry, that $w_1\in \Big$ and $w_2\notin \Big$. 
Since $\prim(w_1)\le D$, we
know that $\tok(x)+\prim(w_2)\ge \frac12D$.  Thus, $x$ gives to $w_2$ charge
$\tok(x)-\frac14D$.  Clearly $x$ finishes with charge at least $\frac14D$.  Furthermore,
$w_2$ ends with charge at least $\prim(w_2)+(\tok(x)-\frac14D)\ge
\frac12D-\frac14D=\frac14D$.

Thus, each vertex of $X$ finishes with charge at least $\frac14D$.  Hence,
$|S\setminus S'|\le |X|\le 6|S|/(D/4)\le 18D/(D/4)=72$, as claimed.  
Combining this with Theorem~\ref{main3} gives $|S|\le \frac52D+72$.
Substituting this bound into the previous inequality gives $|X|\le 6|S|/(D/4)\le
(15D+6\times 72)/(D/4) = 60+(12\times 144)/D$.
\end{clmproof}

Claims~\ref{clm2}--\ref{clm5} show $c=72$ suffices.
When $D>12\times 144$, Claim~\ref{clm5} shows $c=60$ suffices.
\end{proof}

\section{Big Cliques in Squares of Graphs with Mad \texorpdfstring{$\bmm{<4}$}{<4}}
\label{sec4}

In this section, we extend the main result of Section~\ref{sec3} from the class
of 2-degenerate graphs to the class of graphs with maximum average degree less
than 4.  To formalize this idea, we use the following definition.
Fix a positive integer $D$.  Let \Emph{$g(D)$} be the maximum size of a clique in
the square of a graph $G$ with $\mad(G)<4$ and $\Delta(G)\le D$.  Recall, from
Section~\ref{sec3}, that $f(D)$ is defined as the analogous maximum over
2-degenerate graphs.  Since every 2-degenerate graph $G$ has $\mad(G)<4$,
clearly $f(D)\le g(D)$.  In this section we show that also $g(D)-f(D)$ is
bounded by a small constant.  We formalize this result as follows.

\begin{thm}
\label{thm4}
There exists a constant $c$ such that for every positive integer $D$ some
2-degenerate graph $G$ with $\Delta(G)\le D$ has a clique $S'$ in $G^2$ 
with $|S'|\ge g(D)-c$.  In fact, $c=460$ suffices.
\end{thm}

We have made very little effort to minimize $c$ (so it is likely far from
sharp), prefering to present a simpler proof.  Before proving
Theorem~\ref{thm4}, we recall a result of Hocquard, Kim, and
Pierron which slightly simplifies our proof.  If we instead want to avoid using
this result, we can just observe that $G$ is 3-degenerate, so $G^2$ is
$5D$-degenerate, which results in a larger value of $c$.

\begin{thmA}[\cite{HKP}]
If $G$ is a graph with $\Delta(G)\le D$ and $\mad(G)<4$, then $G^2$ has
degeneracy at most $3D$.  In particular, $\omega(G^2)\le 3D+1$.
\end{thmA}

\begin{proof}[Proof of Theorem~\ref{thm4}.]
Fix $D$\aside{$D$, $G$} and let $G$ be a graph with $\Delta(G) \le D$ and $\mad(G) < 4$ such
that $G^2$ contains a clique \Emph{$S$} of order $g(D)$.  Subject to this, choose $G$
to minimize $|V(G)|+|E(G)|$.  As a result, $V(G)\setminus S$ is an independent
set.  We assume that $g(D)\ge \frac52D+60$; otherwise, we are done (for
this choice of $D$) with $c=60$ by Example~\ref{example1}.  Since $\mad(G) <
4$, the graph $G$ has a 3-degeneracy order \EmphE{$\sigma$}{-4mm}.  Let
\Emph{$R_3$} be the subset of $V(G)\setminus S$ each with at least 3
neighbors in $S$.

\begin{clm}
\label{clmA}
$|R_3| < 2|S|$.
\end{clm}
\begin{clmproof}
Let $G_1:=G[S\cup R_3]$.  Clearly $\sum_{v\in V(G_1)}d_{G_1}(v)\ge 6|R_3|$.
So $6|R_3|/(|S|+|R_3|)\le \mad(G_1)\le \mad(G)<4$.  That is,
$6|R_3|<4|S|+4|R_3|$, which proves the claim.
\end{clmproof}

We pass tokens similarly to the way we did in Section~\ref{sec3}, with an added
wrinkle for vertices in $R_3$.  Just before it is deleted, each $v\in S$ passes
a ``primary'' token to each of its neighbors later in $\sigma$.  Again, if a
vertex $v$ is holding $s$ primary tokens just before it is deleted, then $v$
passes $s$ ``secondary'' tokens to each of its neighbors later in $\sigma$ (for
every $s\ge 1$).  Finally, if some $v\in R_3$ has 3 neighbors in $S$ later in
$\sigma$, then $v$ passes a ``tertiary'' token to each of these neighbors.
Again, let $\tok(v)$\aside{$\tok(v)$} (resp.~$\prim(v)$\aaside{$\prim(v)$}{4mm}) denote the number of
tokens (resp.~primary tokens) that each vertex $v$ holds immediately before it 
is deleted (note that $\tok(v)$ now counts \emph{tertiary} tokens, as well as
primary and secondary tokens).  Our next claim is nearly identical to one in the
previous section.

\begin{clm}
\label{clmB}
If $v\in S$ and $w_1,w_2,w_3$ are the neighbors of $v$ later in $\sigma$, if they
exist, then
\begin{equation}
1+\tok(v)+D+\prim(w_1)+\prim(w_2)+\prim(w_3)+12\ge |S|\ge
5D/2+60.\tag{$\star\star$}
\label{key-ineq2}
\end{equation}
\end{clm}

\begin{clmproof}
Fix $v\in S$.  
Recall that $v$ must be adjacent in $G^2$ to every
vertex in $S$.  Each neighbor $w$ of $v$, with $w\notin R_3$, that precedes $v$
in $\sigma$ is the source of at most one such adjacency in $G^2$ to a
vertex $x$ from which $v$ receives no token (neither primary, nor secondary,
nor tertiary); if $w\in R_3$, then it can create two adjacencies for $v$ in
$G^2$, but one of these is accounted for by a tertiary token sent to $v$.
So in total $v$ has at most $D$ such adjacencies \mbox{(that send no tokens
to $v$).}  

Each other adjacency in $G^2$ must
be accounted for either (a) by a token received by $v$ or (b) by a primary token
that has been or will be received by $w_1$ or $w_2$ or $w_3$ or (c) by being an
adjacency to $w_1$ or $w_2$ or $w_3$,
or to a neighbor of some $w_i$ that comes later than $w_i$ in $\sigma$.  At most
$\tok(v)$ vertices of $S$ are handled by (a), and at most $3+3^2=12$ vertices of $S$ are
handled by (c).  The number handled by (b) is at most
$\prim(w_1)+\prim(w_2)+\prim(w_3)$.
This implies~\eqref{key-ineq2}.
\end{clmproof}
\smallskip

Similarly to what we did in Section~\ref{sec3}, we let $\Emph{\Big} := \{v\in V(G):
\prim(v) > \frac18D\}$, we let $\Emph{\Basic}:= \{v\in S: \tok(v) < \frac14D\}$
and we let $\EmphE{\NonBasic}{4mm}:= S\setminus \Basic$.  
If $D\le 2$, then $g(D)\le 5$, and we are done; so instead assume $D\ge 3$.
Note that $|\Big| \le 3|S|/(D/8) \le 3(3D+1)/(D/8)= 72+24/D\le 80$; 
the first inequality here uses Theorem~A.
The total number of primary and secondary tokens is at most
$(3+3^2)|S|$, and the total number of tertiary tokens is at most $3|R_3|<6|S| \le
18D+3$.  So the total number of tokens is, by Claim~\ref{clmA}, at most
$12|S|+3|R_3|<18|S|\le 54D+3$. 
Thus, $|\NonBasic| \le (54D+3)/(D/4) = 216+12/D\le 220$.

\begin{clm}
Every vertex $v\in \Basic$ has at least two neighbors in $\Big$ (that are later in
$\sigma$).
\label{clmC}
\end{clm}
\begin{clmproof}
If $v\in\Basic$, then $\tok(v)<\frac14D$, so~\eqref{key-ineq2} implies
$\prim(w_1)+\prim(w_2)+\prim(w_3) \ge \frac54D$.  By symmetry, assume 
$\prim(w_1)\ge \prim(w_2)\ge \prim(w_3)$.  By Pigeonhole, $\prim(w_1) > \frac18D$, so
$w_1\in \Big$.  Clearly, $\prim(w_1)\le D$.  Thus, $\prim(w_2)+\prim(w_3) \ge
\frac14D$.  Again, by Pigeonhole, $\prim(w_2)\ge \frac18D$, so also
$w_2\in\Big$.
\end{clmproof}

Let $R:=\{v\notin \Big:|N(v)\cap (\Big\cup\Basic)|\ge 3\}$ and $S^+:=\Basic\cup
\Big$.\aside{$R$, $S^+$}

\begin{clm}
We have $\sum_{v\in R}(d_{S^+}(v)-2)<2|\Big|$.
\label{clmD}
\end{clm}
\begin{clmproof}
Let $S^-:=\Basic\setminus(R\cup\Big)$ and $G_2:=G[\Basic\cup R\cup\Big]$.
Note that $\sum_{v\in V(G_2)}d_{G_2}(v) \ge 4|S^-|+2\sum_{v\in R}d_{S^+}(v)$, by
Claim~\ref{clmC}.
Thus, $4>\mad(G_2) \ge (4|S^-|+2\sum_{v\in
R}d_{S^+}(v))/(|S^-|+|R|+|\Big|)$.  So $4|S^-|+2\sum_{v\in
R}d_{S^+}(v)<4(|S^-|+|R|+|\Big|)$, which proves the claim.
\end{clmproof}

Let \Emph{$S'$}$:=\Basic\setminus\Big$.  
(Unlike in Section~\ref{sec3}, here possibly $\Basic\cap\Big\ne\emptyset$.)
\smallskip

We form $G'$, $S'$, $\sigma'$ from $G$, $S$, $\sigma$, as follows.
\smallskip

\begin{enumerate}
\item[(1)] Move $\Big$ to the end of the order, and delete all edges with neither endpoint in $S'$.
\item[(2)] For each vertex $v\notin S'\cup\Big$, delete all but two edges
incident to $v$, and move $v$ to the start of the order; remove all endpoints
of those deleted edges from $S'$.  
\item[(3)] Now consider $v\in S'$; by (2), every
neighbor of $v$ later in the order is in $S'\cup \Big$.  If $v$ has at most two
neighbors later in the order, then do nothing.  Otherwise, delete all but two
edges from $v$ to neighbors later in the order; in particular, delete edges from
$v$ to two big neighbors.  Remove from $S'$ vertex $v$ and any endpoints of
deleted edges that were in $S'$.  
\end{enumerate}

Denote by $G'$ the graph resulting from the process above.
The resulting order $\sigma'$ is a 2-degeneracy order for $G'$,
and $S'$ induces a clique in $G'^2$.  All that remains is to show that $|S'|\ge
|S|-|\NonBasic|-3|\Big| \ge |S|-220-3\times 80 = |S|-460$.

Before (1) above, we have $|S'|\ge |\Basic|-|\Big| = |S|-|\NonBasic|-|\Big|$.
It is straightforward to check that for each vertex $v$ considered in (2) and (3),
all its neighbors later in the order are in $S^+$.
Further, the decrease in $|S'|$
when we process $v$ is at most $\max\{0,d_{S^+}(v)-2\}$.  Thus, the 
total decrease in 
$|S'|$ from (2) and (3)
is at most $\sum_{v\in V(G)\setminus\Big}
\max\{0,d_{S^+}(v)-2\} =\sum_{v\in R}(d_{S^+}(v)-2)<2|\Big|$, by
Claim~\ref{clmD}.  So $|S'|\ge |S|-|\NonBasic|-3|\Big|\ge |S|-460$. 
\end{proof}

\section{Open Questions}

In this paper we have determined, among graphs with maximum degree at most $D$, the
maximum order of a clique in $G^2$, up to an additive constant, (a) when $G$ is
2-degenerate and (b) when $\mad(G)<2\times 2$.  
Recall that $f(D)$ and $g(D)$ are the largest possible sizes of a clique in $G^2$ when
$G$ has maximum degree at most $D$ and satisfies (a) and (b), respectively.
It is natural to ask when $g(D)=f(D)$.

\begin{ques}
For which positive integers $D$ does $g(D)=f(D)$?
\end{ques}

We can ask analogous questions for larger values of degeneracy
or maximum average degree.

\begin{ques}
For each positive integer $k$, what is the minimum value
$\alpha_k$\aside{$\alpha_k$, $c_k$} such that there exists a constant $c_k$
such that if $G$ is $k$-degenerate with maximum degree at most $D$, then
$\omega(G^2)\le \alpha_kD+c_k$?
\end{ques}

Recall that if $G$ is $k$-degenerate with $\Delta(G)\le D$, then $G^2$ has
degeneracy at most $(2k-1)D-k^2$; this is witnessed by any order witnessing
that $G$ is $k$-degenerate.  Thus, $\alpha_k\le 2k-1$.
By extending the ideas in Example~\ref{example2}, we can also construct such
graphs $G$ for which the degeneracy of $G^2$ is equal to $(2k-1)D-k^2$; see
Example~\ref{example3} below.
However, we believe that $\alpha_k<2k-1$.  In fact, it is interesting to ask
about $\lim_{k\to\infty}(2k-1)-\alpha_k$ and $\lim_{k\to\infty}(2k-1)/\alpha_k$.
We suspect the first limit is infinite, but do not have a conjectured value for
the second limit.

Now we consider the larger class of graphs with $\mad(G)<2k$.

\begin{ques}
For each positive integer $k$, what is the minimum value
$\beta_k$\aside{$\beta_k$, $d_k$} such that there exists a constant $d_k$ such
that if $\mad(G)<2k$ and $G$ has maximum degree at most $D$, then
$\omega(G^2)\le \beta_kD+d_k$?
\end{ques}

If $\mad(G)<2k$, then $G$ has degeneracy at most $2k-1$.  Thus, by the argument
above, $\beta_k\le 2(2k-1)-1=4k-3$.  This was significantly improved by
Kierstead et al.~\cite{KYY} who showed that if $\Delta(G)\le D$ and
$2(k-1)<\mad(G)\le 2k$, then $G^2$ has degeneracy at most $(2k-1)D+2k$.
Thus, $\beta_k\le 2k-1$.
In this paper, we proved $\alpha_2=\beta_2=5/2$.  
We believe this equality persists for larger values of $k$, but in general
we have no conjecture for the precise value of $\alpha_k$ (and $\beta_k$).

\begin{conj}
For all $k\ge 2$, we have $\beta_k=\alpha_k$.
\end{conj}

\noindent
In any future efforts to improve the upper bounds on $\alpha_k$ and $\beta_k$, it
would be natural to try a token passing scheme, similar to what we used in
Sections~\ref{sec3} and~\ref{sec4}.

This paper was motivated by a question of Hocquard, Kim, and Pierron: Given a
positive integer $D$, what is the maximum value of $\chi(G^2)$ over all graphs
$G$ with maximum degree $D$ such that (a) $G$ is 2-degenerate or (b)
$\mad(G)<4$.  It is natural to generalize these questions. 

\begin{ques}
\label{HKM-gen-ques}
Fix positive integers $D$ and $k$.  (a) What is the maximum value of $\chi(G^2)$
over all graphs $G$ with maximum degree $D$ that are $k$-degenerate?
(b) What is the maximum value of $\chi(G^2)$ over all graphs $G$ with maximum
degree $D$ such that $\mad(G)<2k$?
\end{ques}

In an effort to attack Question~\ref{HKM-gen-ques}, it would be natural to try
to improve the degeneracy bound for $G^2$,
mentioned above, proved by Kierstead et al.
We end this short section by observing that 
this bound is nearly the best possible.  Essentially,
we generalize Example~\ref{example2} to $k$-degenerate graphs for each larger
$k$.  (Recall that each $k$-degenerate graph $G$ has $\mad(G)<2k$.)
For the proof, we will use the following well-known result of Hajnal and
Szemer\'{e}di~\cite{HS}.

\begin{thmHS}[\cite{HS}]
If $G$ is a graph with $|V(G)|=(k+1)r$ and $\Delta(G)\le k$, for positive
integers $r$ and $k$, then $G$ has a proper coloring with $k+1$ color classes
each~of~size~$r$.
\end{thmHS}

\begin{example}
\label{example3}
For all integers $k$ and $D$ with $D\ge k\ge 2$, we can construct a graph
$H_{D,k}$ that is $k$-degenerate with maximum degree $D$ but such that
$H_{D,k}^2$ is not $((2k-1)D-k^2-1)$-degenerate. 
\smallskip

We essentially generalize Example~\ref{example2}.  We want to partition some
$S\subseteq V(G)$
into parts of size $D$ and repeat this step $k$ times, with no two vertices in
the same part in more than one step.  Afterward, we want to partition $S$
into parts of size $k$, repeated $D-k$ times, such that each pair of vertices
appears in a common part (over all $k+(D-k)$ times) at most once.  In fact, all
of this can be done by repeated application of the Hajnal--Szemer\'{e}di
Theorem.

Fix a set $S$ s.t. $|S|=2kD^2$.
We will build a graph $G$ with $S\subseteq V(G)$ and also build an auxiliary
graph $J$ with $V(J)=S$, which will aid in the construction of $G$.
(At the end, we will let $H_{D,k}:=G$.)
Initially, let $V(G)=V(J)=S$ and $E(G)=E(J)=\emptyset$.  Trivially, we can partition $V(J)$ into
$|S|/D$ independent sets, each of size $D$.  For each part $P_i$ in this partition,
add to $G$ a vertex $w_i$ adjacent to precisely the vertices in $P_i$.
Moreover, add to $E(J)$ the edge $vv'$ for each pair $v,v'$ in a common part. 
Now we again partition $V(J)$ into $|S|/D$ independent sets, each of size $D$;
this time we use Hajnal--Szemer\'{e}di.  We repeat this process, partitioning
$V(J)$ a total of $k$ times, each time into $|S|/D$ parts, and each inducing no
edges.  After each partition, we add the prescribed edges to $E(J)$ and $E(G)$.
This ensures that each pair of vertices in $S$ appear in a common part in at
most one partition.  In a second phase we partition $|S|$ into
parts of size $k$, and we do this $D-k$ times.  After the final partition, the
degree of each vertex in $J$ is exactly $k(D-1)+(D-k)(k-1)\le D(2k-1)$.
Thus, $H^2_{D,k}[S]$ has minimum degree $k(D-1)+(D-k)(k-1)$.
We do not really care about the number of color classes, just that the size of
each is $D$ (in the first phase) or $k$ (in the second phase).  
Although we required above that $|S|=2kD^2$, we note that the same construction works
whenever $|S|\ge 2kD^2$ and $kD\big| |S|$.

Let $H_{D,k}:=G$.  We verify, as follows, that $H_{D,k}$ is $k$-degenerate.
Each vertex $w_i$ arising from a part of size $k$ has degree exactly $k$.  So
all of these $w_i$ can be deleted first.  Next, we can delete all vertices of $S$,
since each of them has exactly $k$ remaining neighbors.  Finally, we can delete
all vertices $w_i$ arising from parts of size $D$ (which form an independent set).
\exampleEnd
\end{example}


\end{document}